\documentclass{birkjour}
 \newtheorem{thm}{Theorem}[section]
 
 \newtheorem{lem}[thm]{Lemma}
 
 \theoremstyle{definition}
 \newtheorem{defn}[thm]{Definition}
 \theoremstyle{remark}
 \newtheorem{rem}[thm]{Remark}
 \newtheorem*{ex}{Example}
 \numberwithin{equation}{section}

\begin{document}

\title[Riemann Solitons and Almost Riemann Solitons]
 {Riemann Solitons and Almost Riemann Solitons on Almost Kenmotsu Manifolds}

\author[Venkatesha, H.A. Kumara, D.M. Naik]{V. Venkatesha, H. Aruna Kumara and Devaraja Mallesha Naik}

\address{
Department of Mathematics\\
Kuvempu University\\
Shankaraghatta\\
Karnataka 577 451\\
India}

\email{vensmath@gmail.com} 


\email{arunmathsku@gmail.com}

\email{devarajamaths@gmail.com.}



\begin{abstract}
The aim of this article is to study the Riemann soliton and gradient almost Riemann soliton on certain class of almost Kenmotsu manifolds. Also some suitable examples of Kenmotsu and $(\kappa,\mu)'$-almost Kenmotsu manifolds are constructed to justify our results.
\end{abstract}

\maketitle
\section{Introduction}
Hamilton \cite{Hamilton} introduced the concept of Ricci flow. The idea of Ricci flow generalized to the concept of Riemann flow (see \cite{Udriste,Udriste2}). As an analog of Ricci soliton, Hiri\u{c}a and Udri\c{s}te \cite{Hirica} introduced and studied Riemann soliton. This is appears as a self-similar solution of the Riemann flow \cite{Udriste}:
\begin{align*}
\frac{\partial}{\partial t} G(t)=-2R(g(t)),\quad t\in[0,I],
\end{align*}
where $G=\frac{1}{2}g\mathbin{\bigcirc\mspace{-15mu}\wedge\mspace{3mu}} g$, $R$ is the Riemann curvature tensor associated to the metric $g$ and $\mathbin{\bigcirc\mspace{-15mu}\wedge\mspace{3mu}}$ is Kulkarni-Nomizu product. These extensions are natural, since some results in the Riemann flow resembles the case of Ricci flow (for detail, see \cite{Udriste2}). For instance, the Riemann flow satisfies the short time existence and uniqueness \cite{Hirica2}. Further the authors in \cite{Step} characterize the Riemann soliton in terms of infinitesimal harmonic transformation. For $(0,2)$-tensors $A$ and $B$, thier Kulkarni-Nomizu product $A\mathbin{\bigcirc\mspace{-15mu}\wedge\mspace{3mu}}B$ is given by 
\begin{align}\label{1.1}
\nonumber (A\mathbin{\bigcirc\mspace{-15mu}\wedge\mspace{3mu}} B)(X_1,X_2,X_3,&X_4)=A(X_1,X_3)B(X_2,X_4)+A(X_2,X_4)B(X_1,X_3)\\
&-A(X_1,X_4)B(X_2,X_3)-A(X_2,X_3)B(X_1,X_4).
\end{align}
Riemann soliton is a smooth manifold $M$ together with Riemannian metric $g$ that satisfies
\begin{align}\label{1.2}
2R+\lambda (g\mathbin{\bigcirc\mspace{-15mu}\wedge\mspace{3mu}}g)+(g\mathbin{\bigcirc\mspace{-15mu}\wedge\mspace{3mu}}\pounds_V g)=0,
\end{align}
where $V$ is a vector field called as a potential vector field, $\pounds_V$ denotes the Lie-derivative and $\lambda$ is a constant. The Riemann soliton also corresponds as a fixed point of the Riemann flow, and they can be viewed as a dynamical system, on the space of Riemannian metric modulo diffeomorphism. Note that notion of Riemann soliton generalizes the space of constant sectional curvature (i.e., $R=kg\mathbin{\bigcirc\mspace{-15mu}\wedge\mspace{3mu}}g$, for some constant $k$). A Riemann soliton is called shrinking when $\lambda<0$, steady when $\lambda=0$ and expanding $\lambda>0$. If the vector field $V$ is the gradient of a potential function $u$, then we get the notion of gradient Riemann soliton. In such a case, the equation \eqref{1.2} transform into
\begin{align}\label{1.3}
R+\frac{1}{2}\lambda(g\mathbin{\bigcirc\mspace{-15mu}\wedge\mspace{3mu}}g)+(g\mathbin{\bigcirc\mspace{-15mu}\wedge\mspace{3mu}}Hess~u)=0,
\end{align}
where $Hess~u$ denotes the Hessian of the smooth function $u$. In \cite{Hirica}, it is proved that a Riemann soliton on a compact Riemannian manifold is a gradient. A Riemann soliton is called trivial when potential vector field $V$ vanishes identically, and this case manifold is of constant sectional curvature. If $\lambda$ appearing in equation \eqref{1.2} and \eqref{1.3} is a smooth function, then $g$ is called almost Riemann soliton and almost gradient Riemann soliton respectively. 

\par In \cite{Hirica}, Hiri\u{c}a and Udri\c{s}te studied Riemann soliton and gradient Riemann soliton within the framework of Sasakian manifold. In their study, it is proved that if Sasakian manifold admits a Riemann soliton whose soliton vector field $V$ is pointwise collinear with $\xi$ (or a gradient Riemann soliton and potential function is harmonic), then it is Sasakian space form. In this study, the authors exposes an open problem as classify gradient Riemann soliton for arbitrary potential function on Sasakian manifold. To fulfill this classification, the present authors in \cite{Dev} consider a Riemann soliton on contact manifold, and obtained several intersting results. These results on contact geometry intrigues us to study the Riemann soliton on other almost contact metric manifolds. In this paper, we classify certain class of almost Kenmotsu manifold which admits a Riemann soliton and almost gradient Riemann solion.

\section{Preliminaries}
Let $(M,g)$ be a $(2n+1)$-dimensional smooth Riemannian manifold. On this manifold if there exists a (1,1)-type tensor field $\varphi$, a global vector field $\xi$ and 1-form $\eta$ such that
\begin{align}
\label{2.1}\varphi^2 X=-X+\eta(X)\xi, \quad \eta(\xi)=1,\\
g(\varphi X,\varphi Y)=g(X,Y)-\eta(X)\eta(Y),
\end{align}
for any vector field $X,Y$ on $M$, then we say that $(\varphi,\xi,\eta,g)$ is an almost contact metric structure and $M$ is an almost contact metric manifold \cite{Blair}. Generally, $\xi$ and $\eta$ are called Reeb or characterstic vector field and almost contact 1-form respectively. We remark that an almost contact metric structure on a Riemannian manifold $M$ may be regarded as a reduction of the structure group $M$ to $U(n)\times 1$. For such a manifold, we define a fundamental 2-form $\Phi$ by $\Phi(X,Y)=g(X,\varphi Y)$. It is well known that the normality of an almost contact structure is expressed by vanishing of the tensor $N_\varphi=[\varphi,\varphi]+2d\eta\otimes\xi$, where $[\varphi,\varphi]$ is the Nijenhuis tensor of $\varphi$ (see \cite{Blair}).
\par Almost Kenmotsu manifold \cite{Janssens Vanhecke} is an almost contact metric manifold such that $d\eta=0$ and $d\Phi=2\eta\wedge\Phi$. A normal almost Kenmotsu manifold is called Kenmotsu manifold, and this normality condition is expressed as
\begin{align*}
(\nabla_X \varphi)Y=g(\varphi X,Y)\xi-\eta(Y)\varphi X,
\end{align*}
for any vector field $X,Y$ on $M$, where $\nabla$ denotes the Levi-Civita connection of the Riemannian metric $g$. On an almost Kenmotsu manifold, we set $\ell=R(\cdot,\xi)\xi$, $2h=\pounds_\xi \varphi$ and $h'=h\circ\varphi$, where $R$ denotes the curvature tensor of $M$ and $\pounds$ is the Lie-derivative. These mentioned tensor fields plays important role in almost Kenmotsu manifold, and it is easily seen that both are symmetric and satisfies the following equations
\begin{align}
\label{2.3}\nabla_X \xi=X-\eta(X)\xi+h'X,\\
\label{2.4}h\xi=\ell\xi=0,\quad trh=trh'=0,\quad h\varphi+\varphi h=0,\\
tr\ell=S(\xi,\xi)=g(Q\xi,\xi)=-2n-trh^2,
\end{align}
where $S$ is the  Ricci curvature tensor, $Q$ is the Ricci operator with respect to $g$ anb $tr$ is the trace operator. 

\begin{defn}
	On almost contact metric manifold $M$, a vector field $V$ is said to be infinitesimal
	contact transformation if $\pounds_V\eta=\sigma\eta$, for some function $\sigma$. In particular, we called $V$ is a strict infinitesimal contact transformation if $\pounds_V\eta=0$.
\end{defn}
Now, we recall some formulas which are helpfull to prove our main results. \\
Using the symmetry of $\pounds_V\nabla$ in the commutation formula (see Yano \cite{Yano}):
\begin{align*}
(\pounds_V&\nabla_X g-\nabla_X\pounds_V g-\nabla_{[V,X]}g)(Y,Z)\\
&=-g((\pounds_V\nabla)(X,Y),Z)-g((\pounds_V\nabla)(X,Z),Y),
\end{align*}
we derive
\begin{align}
\label{2.6}2g((\pounds_V\nabla)(X,Y),Z)=(\nabla_X\pounds_V g)(Y,Z)+(\nabla_Y\pounds_Vg)(Z,X)-(\nabla_Z\pounds_Vg)(X,Y).
\end{align}
The following well-known commutation equations is also known to us from Yano \cite{Yano}:
\begin{align}
\label{2.7}(\pounds_V R)(X,Y)Z=(\nabla_X\pounds_V\nabla)(Y,Z)-(\nabla_Y\pounds_V\nabla)(X,Z),\\
\label{2.8}\pounds_V\nabla_X Y-\nabla_X\pounds_VY-\nabla_{[V,X]}Y=(\pounds_V\nabla)(X,Y).
\end{align}

\section{Riemann soliton and almost gradient Riemann soliton on Kenmotsu manifolds}
Dileo and Pastore \cite{Dileo Pastore} proved that the almost contact metric structure of an almost Kenmotsu manifold is normal if and only if the foliations of the distribution $\mathcal{D}$ are K\"ahlerian and the $(1,1)$-type tensor field $h$ vanishes. In particular, we have immediately that almost Kenmotsu manifold is a Kenmotsu manifold if and only if $h=0$. The following formulae are holds for a Kenmotsu manifold \cite{Kenmotsu}
\begin{align}
\label{3.1}&\nabla_X \xi=X-\eta(X)\xi,\\
\label{3.2}&R(X,Y)\xi=\eta(X)Y-\eta(Y)X,\\
\label{3.3}&Q\xi=-2n\xi.
\end{align}
Differentiating \eqref{3.3} along an arbitrary vector field $X$ and recalling \eqref{3.1} we deduce
\begin{align}
\label{3.4}(\nabla_X Q)\xi=-QX-2nX.
\end{align}
In this section, we aim to investigate the existence of geometry of Riemann soliton and an almost gradient Riemann soliton on Kenmotsu manifold. First, we prove the following result.

\begin{lem}
	Let $M$ be a $(2n+1)$-dimensional Kenmotsu manifold. If $(g,V)$ is a Riemann soliton with soliton vector $V$ has a constant divergence, then the potential vector field and Ricci tensor satisfy
	\begin{align}
	\label{3.5}&div V=2n(1-\lambda),\\
	\label{3.6}&(\pounds_V S)(X,\xi)=\frac{1}{2n-1}\{-(Yr)+(\xi r)\eta(Y)\},
	\end{align}
	where $div$ denotes the divergence operator.
\end{lem}

\begin{proof}
	As a result of \eqref{1.1}, the Riemann soliton equation \eqref{1.2} can be expressed as
	\begin{align}
	\nonumber 2R(X,Y,&Z,W)+2\lambda\{g(X,W)g(Y,Z)-g(X,Z)g(Y,W)\}\\
	\nonumber &+\{g(X,W)(\pounds_V g)(Y,Z)+g(Y,Z)(\pounds_V g)(X,W)\\
	\label{3.7}&-g(X,Z)(\pounds_V g)(Y,W)-g(Y,W)(\pounds_V g)(X,Z)\}=0.
	\end{align}
	Contractiong \eqref{3.7} over $X$ and $W$, we obtain
	\begin{align}
	\label{3.8}(\pounds_V g)(Y,Z)+\frac{2}{2n-1}S(Y,Z)+\frac{2}{2n-1}(2n\lambda+div V)g(Y,Z)=0.
	\end{align}
	Taking covariant derivative of the above relation gives
	\begin{align}
	(\nabla_X \pounds_V g)(Y,Z)+\frac{2}{2n-1}(\nabla_X S)(Y,Z)=0,
	\end{align}
	where we applied $V$ has a constant divergence. Inserting the foregoing equation into \eqref{2.6} we obtain
	\begin{align}
	\label{3.10}g((\pounds_V\nabla)(X,Y),Z)=\frac{1}{2n-1}\{(\nabla_Z S)(X,Y)-(\nabla_X S)(Y,Z)-(\nabla_Y S)(Z,X)\}.
	\end{align}
	The present authors in \cite{Vens} and Ghosh in \cite{Ghosh} independently gave the complete proof of the following formula which is valid for any Kenmotsu manifold
	\begin{align}
	\label{3.11}(\nabla_\xi Q)X=-2QX-4nX.
	\end{align}
	Replacing $Y$ by $\xi$ in \eqref{3.10}, applying \eqref{3.4} and \eqref{3.11} we get
	\begin{align}
	\label{3.12}(\pounds_V\nabla)(X,\xi)=\frac{2}{2n-1}QX+\frac{4n}{2n-1}X.
	\end{align}
	Differentiating \eqref{3.12} covariantly along $Y$ and utilizing \eqref{3.1}, we find
	\begin{align*}
	(\nabla_Y\pounds_V\nabla)(X,\xi)+(\pounds_V\nabla)(X,Y)-\frac{2}{2n-1}\eta(Y)\{QX+2nX\}=\frac{2}{2n-1}(\nabla_Y Q)X.
	\end{align*} 
	Making use of this in the well known expression \eqref{2.7}, we have
	\begin{align}
	\nonumber (\pounds_V R)(X,Y)\xi=&\frac{2}{2n-1}\{\eta(X)QY-\eta(Y)QX+(\nabla_X Q)Y-(\nabla_Y Q)X\}\\
	\label{3.13}&+\frac{4n}{2n-1}\{\eta(X)Y-\eta(Y)X\}.
	\end{align}
	We insert $Y=\xi$ in the above relation and utilizing \eqref{3.4} and \eqref{3.11} to achieve $(\pounds_V R)(X,\xi)\xi=0$. On the other hand, operating $\pounds_V$ to the formula $R(X,\xi)\xi=-X+\eta(X)\xi$ (follows from \eqref{3.2}) yields
	\begin{align*}
	(\pounds_V R)(X,\xi)\xi+g(X,\pounds_V\xi)\xi-2\eta(\pounds_V\xi)X=\{(\pounds_V\eta)X\}\xi,
	\end{align*}
	which by virtue of $(\pounds_V R)(X,\xi)\xi=0$ becomes
	\begin{align}
	\label{3.14}g(X,\pounds_V\xi)\xi-2\eta(\pounds_V\xi)X=\{(\pounds_V\eta)X\}\xi.
	\end{align}
	With the aid of \eqref{3.3}, the equation \eqref{3.8} takes the form
	\begin{align}
	\label{3.15}(\pounds_Vg)(X,\xi)=-\frac{2}{2n-1}\{2n\lambda+div V-2n\}\eta(X).
	\end{align} 
	Now, Lie-differentiating $\eta(X)=g(X,\xi)$ and $g(\xi,\xi)=1$ along $V$ and taking into account \eqref{3.15} provides
	\begin{align*}
	(\pounds_V\eta)X-g(X,\pounds_V\xi)+\frac{2}{2n-1}\{2n\lambda+div V-2n\}\eta(X)=0,\\
	\eta(\pounds_V\xi)=\frac{1}{2n-1}\{2n\lambda+div V-2n\}.
	\end{align*}
	Utilizing these equations in \eqref{3.14} yields $(2n\lambda+divV-2n)\{X-\eta(X)\xi\}=0$. Tracing this provides \eqref{3.5}. Now contracting the equation \eqref{3.13} provides
	\begin{align}
	\label{3.16}(\pounds_V S)(Y,\xi)=\frac{1}{2n-1}\{-(Yr)-2(r+2n(2n+1))\}\eta(Y),
	\end{align}
	where we applied the well-known formula $div Q=\frac{1}{2}grad\,\, r$ and $tr\nabla Q=grad\,\, r$. Taking trace of \eqref{3.11} provides $(\xi r)=-2(r+2n(2n+1))$. By virtue of this and \eqref{3.16} gives \eqref{3.6}. This completes the proof.
\end{proof}

\begin{thm}\label{t3.2}
	Let $M$ be an $\eta$-Einstein Kenmotsu manifold of dimension higher than $3$. If $(g,V)$ is a Riemann soliton with $divV=constant$, then $M$ is Einstein.
\end{thm}
\begin{proof}
	We know that a Kenmotsu manifold $M$ is $\eta$-Einstein if and only if 
	\begin{align}
	\label{3.17}S(X,Y)=\left(\frac{r}{2n}+1\right)g(X,Y)-\left(\frac{r}{2n}+2n+1\right)\eta(X)\eta(Y).
	\end{align}
	With the aid of the above equation, it has been proved by present authors in
	(\cite{Vens}, Lemma 3.4) that on $\eta$-Einstein Kenmotsu manifold of $dim>3$ there holds
	\begin{align*}
	Dr=(\xi r)\xi.
	\end{align*}
	Making use of this in \eqref{3.6}, we have $(\pounds_V S)(X,\xi)=0$. Now, applying $\pounds_V$ to \eqref{3.3}, recalling \eqref{3.17}, \eqref{3.15}, \eqref{3.5} and $\eta(\pounds_V\xi)=0$, we obtain
	\begin{align}
	\label{3.18}(r+2n(2n+1))\pounds_V\xi=0.
	\end{align}
	Suppose that $r\neq2n(2n+1)$ in some open set $\mathcal{O}$ of M. Then on $\mathcal{O}$, $\pounds_V\xi=0=\pounds_V\eta$. Replacing $Y$ by $\xi$ in \eqref{2.8} and using $\pounds_V\xi=0=\pounds_V\eta$, we have
	\begin{align*}
	(\pounds_V\nabla)(X,\xi)=&\pounds_VX-\eta(X)\pounds_V\xi-\{(\pounds_V\eta)X\}\xi\\
	&-\eta(\pounds_V X)\xi-\pounds_VX+\eta(\pounds_V X)\xi=0.
	\end{align*}
	By virtue of this in \eqref{3.12}, we have $QX=-2nX$. Taking trace of this gives $r=-2n(2n+1)$ on $\mathcal{O}$, which is a contradiction on $\mathcal{O}$. Thus, equation \eqref{3.18} gives $r=-2n(2n+1)$ and therefore we can conclude from
	$\eta$-Einstein condition \eqref{3.17} that $M$ is Einstein. 
\end{proof}
It is known that any $3$-dimensional Kenmotsu manifold is $\eta$-Einstein. As a result of Theorem~\ref{t3.2}, it is interesting to study Riemann soliton in Kenmotsu $3$-manifold and here, we prove the following outcome.
\begin{thm}\label{t3.3}
	Let be a 3-dimensional Kenmotsu manifold. If $(g,V)$ represents a Riemann soliton with soliton vector $V$ has a constant divergence, then $M$ is of constant negative curvature $-1$.
\end{thm}
\begin{proof}
It is well known that the Riemannian curvature tensor of $3$-dimensional Riemannian manifold is given by
\begin{align}
\nonumber R(X,Y)Z=&g(Y,Z)QX-g(X,Z)QY+g(QY,Z)X-g(QX,Z)Y\\
\label{n}&-\frac{1}{2}\{g(Y,Z)X-g(X,Z)Y\}.
\end{align}
Replacing $\xi$ in place of $Y$ and $Z$ in above equation, employing \eqref{3.2} and \eqref{3.3} gives
\begin{align*}
QX=\left(\frac{r}{2}+1\right)X-\left(\frac{r}{2}+3\right)\eta(X)\xi,
\end{align*}
which is equivalent to
\begin{align}
\label{n1}S(X,Y)=\left(\frac{r}{2}+1\right)g(X,Y)-\left(\frac{r}{2}+3\right)\eta(X)\eta(Y).
\end{align}
Applying $\pounds_V$ to \eqref{3.3}, recalling \eqref{n1} yields 
\begin{align*}
(\pounds_V S)(X,\xi)+\left(\frac{r}{2}+1\right)g(X,\pounds_V\xi)-\left(\frac{r}{2}+3\right)\eta(X)\eta(\pounds_V\xi)=-2(\pounds_V\eta)(X).
\end{align*} 
This together with \eqref{3.6} provides
\begin{align}
\nonumber-Y(r)+\xi(r)\eta(Y)+\left(\frac{r}{2}+1\right)g(X,\pounds_V\xi)\\
\label{n2}-\left(\frac{r}{2}+3\right)\eta(X)\eta(\pounds_V\xi)+2(\pounds_V\eta)(X)=0.
\end{align}
By virtue of \eqref{3.5} and \eqref{3.15}, it follows from \eqref{n2} that
\begin{align}
\label{n3}(r+6)g(X,\pounds_V\xi)=2\{Y(r)-\xi(r)\eta(Y)\},
\end{align}
where we used Theorem~4.1 of Wang \cite{Wang}. Now suppose that $r=-6$. Then from \eqref{n1} we have $QX=-2X$, and substituting this in \eqref{n} one gets
\begin{align*}
R(X,Y)Z=-\{g(Y,Z)X-g(X,Z)Y\},
\end{align*}
which means $M$ is of constant curvature $-1$.
\par Now we suppose that $r\neq-6$ in some open set $\mathcal{O}$ of $M$. Then on $\mathcal{O}$, the relation \eqref{n3} can be written as
\begin{align}
\label{n4}\pounds_V\xi=f\{Dr-\xi(r)\xi\},
\end{align}
where $f=\frac{2}{r+6}$. Taking $Y$ by $\xi$ in \eqref{2.8} and using \eqref{3.1}, \eqref{n4} and \eqref{3.12}, it follows that
\begin{align*}
\nonumber f\{&X(r)\eta(Y)+Y(r)\eta(X)-2\xi(r)\eta(X)\eta(Y)+g(\nabla_X Dr,Y)-X(\xi(r))\eta(Y)\}\\
&+X(f)\{Y(r)-\xi(r)\eta(Y)\}+\left\{f~\xi(r)-\frac{1}{2}\xi(r)\right\}g(\varphi X,\varphi Y)=0.
\end{align*}
Antisymmetrizing the foregoing equation and keeeping in mind the Poincare lemma: $g(\nabla_X Dr,Y)=g(\nabla_Y Dr,X)$, we find
\begin{align*}
f\{Y(\xi(r))\eta(X)-X(\xi(r))\eta(Y)\}+X(f)\{Y(r)-\xi(r)\eta(Y)\}\\
-Y(f)\{X(r)-\xi(r)\eta(X)\}=0.
\end{align*}
Putting $Y=\xi$ in the above equation and using Theorem~4.1 of Wang \cite{Wang} one get
\begin{align*}
(2f-\xi(f))\{X(r)-\xi(r)\eta(X)\},
\end{align*}
which implies
\begin{align}
\label{n5}(2f-\xi(f))\{Dr-\xi(r)\xi\}.
\end{align}
In the following first case we show that $Dr=\xi(r)\xi$, one get $g(Dr,\xi)=-2(r+6)$. Since $f=\frac{2}{r+6}$ we obtain
$\xi(f)=2f$. In the second case we show $\xi(f)=2f$ implies $Dr=\xi(r)\xi$. Thus for any point $p\in\mathcal{O}$, we have $(Dr)_p= \xi_p(r)\xi_p$ if and only if $\xi_p(f)=2f$. Hence from \eqref{n5}, we have either $\xi(f)=2f$ or $Dr=\xi(r)\xi$.\\ \\
\textit{Case~1}: First we assume
\begin{align}
Dr=\xi(r)\xi, 
\end{align}
and so \eqref{n4} becomes $\pounds_V\xi=0$ on the open set $\mathcal{O}$ of $M$. Setting $Y$ by $\xi$ in \eqref{2.8} and using $\pounds_V\xi=0$, we have $(\pounds_V \nabla)(X,\xi)=0$. This together with \eqref{3.12} provides $QX=-2X$. Contracting this with respect to $X$ gives $r=-6$ which yields a contradiction.\\ \\ 
\textit{Case~2}: Suppose if $\xi(f)=2f$, then using $f=\frac{2}{r+6}$ we obtain $\xi(r)=-2(r+6)$  which means
\begin{align*}
g(Dr,\xi)=-2(r+6).
\end{align*}
As $r\neq -6$ on $\mathcal{O}$, the above equation implies $Dr=f\xi$, for some smooth function $f$. In fact, we have $Dr=\xi(r)\xi$ and so following the \textit{Case~1} we arrive the contradiction. This completes the proof.
\end{proof}

Now, we consider a Kenmotsu metric as a gradient almost Riemann soliton and first, we need the following result to prove our main theorem.
\begin{lem}
	For gradient almost Riemann soliton, the following formula is valid
	\begin{align}
	\nonumber R(X,Y)Du=&\frac{1}{2n-1}\{(\nabla_Y Q)X-(\nabla_X Q)Y\}\\
	\label{3.19}&+\frac{1}{2n-1}\{Y(2n\lambda+\Delta u)X-X(2n\lambda+\Delta u)Y\},
	\end{align}
	where $\Delta u=div Du$, $\Delta$ is the Laplacian operator.
\end{lem}

\begin{proof}
	Contracting gradient almost Riemann soliton equation \eqref{1.3}, we get
	\begin{align*}
	Hess~u+\frac{1}{2n-1}S+\frac{1}{2n-1}(2n\lambda+\Delta u)g=0.
	\end{align*}
	Note that the above equation may exhibited as
	\begin{align}
	\label{3.20}\nabla_Y Du=-\frac{1}{2n-1}QY-\frac{1}{2n-1}(2n\lambda+\Delta u)Y.
	\end{align}
	By straightforward computations, using the well known expression of the curvature tensor:
	\begin{align*}
	R(X,Y)=\nabla_X\nabla_Y-\nabla_Y\nabla_X-\nabla_{[X,Y]},
	\end{align*}
	and the repeated use of equation \eqref{3.20} gives the desired result.
\end{proof}

\begin{thm}\label{t3.5}
	If metric of a Kenmotsu manifold $M$ of dimeninsion $(2n+1)$ represents a gradient almost Riemann soliton, then either $M$ is Einstein or the soliton vector field $V$ is pointwise collinear with the characteristic vector field $\xi$ on an open set $\mathcal{O}$ of $M$. In the first case, if $M$ is complete, then $M$ is locally isometric to a hyperbolic space $\mathbb{H}^{2n+1}$, and the function $2n\lambda+\Delta u$, upto an additive constant, can be expressed as a linear combination of $cosh\,t$ and $sinh\,t$.
\end{thm}

\begin{proof}
	Replacing $Y$ by $\xi$ in \eqref{3.19} and employing \eqref{3.4} and \eqref{3.11}, we find
	\begin{align}
	\label{3.21}R(X,\xi)Du=-\frac{1}{2n-1}\{QX+2nX\}+\frac{1}{2n-1}\{\xi(2n\lambda+\Delta u)X-X(2n\lambda+\Delta u)\xi\}.
	\end{align} 
	From \eqref{3.2}, we have $R(X,\xi)Y=g(X,Y)\xi-\eta(Y)X$. Employing this into \eqref{3.21} provides
	\begin{align}
	\nonumber g(X,Du+\frac{1}{2n-1}D(2n\lambda+\Delta u))&\xi-\{\xi(u)+\frac{1}{2n-1}\xi(2n\lambda+\Delta u)\}X\\
	\label{3.22}&=-\frac{1}{2n-1}\{QX+2nX\}.
	\end{align}
	Inner product of the previous equation with $\xi$ and using \eqref{3.3} provides $X(u+\frac{1}{2n-1}(2n\lambda+\Delta u))=\xi(u+\frac{1}{2n-1}(2n\lambda+\Delta u))\eta(X)$, from which we have
	\begin{align}
	\label{3.23}d(u+\frac{1}{2n-1}(2n\lambda+\Delta u))=\xi(u+\frac{1}{2n-1}(2n\lambda+\Delta u))\eta,
	\end{align}		
	where $d$ is the exterior derivative. This means that $u+\frac{1}{2n-1}(2n\lambda+\Delta u)$ is invariant along
	the distribution $\mathcal{D}$, i.e., $X(u+\frac{1}{2n-1}(2n\lambda+\Delta u))=0$ for any vector field $X\in\mathcal{D}$. Utilizing \eqref{3.23} into \eqref{3.22}, we obtain
	\begin{align}
	\nonumber \{\xi(u)+\frac{1}{2n-1}\xi(2n\lambda+\Delta u)\}\eta(X)&\xi-\{\xi(u)+\frac{1}{2n-1}\xi(2n\lambda+\Delta u)\}X\\
	\label{3.24}&=-\frac{1}{2n-1}\{QX+2nX\}.
	\end{align} 
	Contracting the above equation, one immediately obtain
	\begin{align}
	\label{3.25}\xi(u+\frac{1}{2n-1}(2n\lambda+\Delta u))=\frac{1}{2n-1}\{\frac{r}{2n}+2n+1\}.
	\end{align}
	Making use of last equation in \eqref{3.24} one can obtain $\eta$-Einstein condition \eqref{3.17}. Now, we contract equation \eqref{3.19} over $X$ to deduce
	\begin{align*}
	S(Y, Du)=\frac{1}{2(2n-1)}Y(r)+\frac{2n}{2n-1}Y(2n\lambda+\Delta u).
	\end{align*}
	In the contrast of above equation with \eqref{3.17} we obtain
	\begin{align}
	\nonumber(r+2n)Y(u)-(&r+2n(2n+1))\xi(u)\eta(Y)-\frac{1}{2(2n-1)}Y(r)\\
	\label{3.26} &-\frac{4n^2}{2n-1}Y(2n\lambda+\Delta u)=0.
	\end{align}
	Inserting $Y=\xi$ in \eqref{3.26} and recalling \eqref{3.25} one can get $\xi(r)=-2\{r+2n(2n+1)\}$. Action of $d$ on \eqref{3.23}, we get $dr\wedge\eta=0$, where we used $d^2=0$ and $d\eta=0$. Hence by virtue of $\xi(r)=-2\{r+2n(2n+1)\}$ we have
	\begin{align}
	\label{3.27} Dr=-2\{r+2n(2n+1)\}\xi.
	\end{align}
	 Suppose that $X$ in \eqref{3.26} is orthogonal to $\xi$. Taking into account $(u+\frac{1}{2n-1}(2n\lambda+\Delta u))$ being a constant along $\mathcal{D}$ and utilizing \eqref{3.23} and \eqref{3.27}, then we get $(r+2n(2n+1))X(u)=0$ for any $X\in\mathcal{D}$. This implies that
	 \begin{align}
	 \label{3.28}(r+2n(2n+1))(Du-\xi(u)\xi)=0.
	 \end{align}
	 \par If $r=-2n(2n+1)$, then the equation \eqref{3.17} shows that $QX=-2nX$, and hence $M$ is Einstein. Since $r=-2n(2n+1)$, it follows from \eqref{3.25} that $\xi(u)=-\frac{1}{2n-1}\xi(2n\lambda+\Delta u)$ and therefore $Du=-\frac{1}{2n-1}D(2n\lambda+\Delta u)\xi$. Thus, equation \eqref{3.20} can be exhibited as
	 \begin{align}
	 \label{3.29}\nabla_X D(2n\lambda+\Delta u)=((2n\lambda+\Delta u)-2n)X.
	 \end{align}
	 According to Theorem~2 of Tashiro \cite{Tashiro} we obtain $M$ is locally isometric to the hyperbolic space $\mathbb{H}^{2n+1}$, when $M$ is complete. Since $\nabla_\xi \xi=0$ and $g(\xi,D(2n\lambda+\Delta u))=\xi(2n\lambda+\Delta u)$, we deduce from \eqref{3.29} that $\xi(\xi(2n\lambda+\Delta u))=(2n\lambda+\Delta u)-2n$. But, it is known \cite{Kenmotsu} that a $(2n+1)$-dimensional Kenmotsu manifold is locally isometric to the warped product $I\times_{ce^t} N^{2n}$, where $N$ is a K\"ahler manifold and $I$ is an open interval. Employing $\xi=\frac{\partial}{\partial t}$ (where $t$ denotes the coordinate on $I$) into \eqref{3.29} we obtain
	 \begin{align*}
	 \frac{d^2(2n\lambda+\Delta u)}{dt}=(2n\lambda+\Delta u)-2n.
	 \end{align*}
	 The solution of this can be exhibited as $(2n\lambda+\Delta u)=A\, cosh\,\,t+B\, sinh\,\,t+2n$, where $A$, $B$ constants on $M$.
	 \par Now, suppose that $r\neq -2n(2n+1)$ in some open set $\mathcal{O}$ of $M$, then from \eqref{3.28} we have $Du=\xi(u)\xi$, and this completes the proof.
 \end{proof}

\begin{rem}
	Theorem~1 in \cite{Ghosh2}, Theorem~1 in \cite{Ghosh3} and Theorem~3 in \cite{Ghosh} of Ghosh is a direct corollary of Theorem \ref{t3.2}, Theorem \ref{t3.3} and Theorem \ref{t3.5} respectively.
\end{rem}

Now we construct an example of Riemann soliton on 3-diemnsional Kenmotsu manifold, and verify our results.
\begin{ex}
	Let us indicate the canonical coordinates on $\mathbb{R}^3$ by $(x,y,z)$, and take
	the 3-dimensional manifold $M\subset\mathbb{R}^3$ defined by
	\begin{align*}
	M=\{(x,y,z)\in\mathbb{R}^3|z\neq0\}.
	\end{align*}
	We may easily verify that putting
	\begin{align*}
	\varphi\left(\frac{\partial}{\partial x}\right)=\frac{\partial}{\partial y},\quad \varphi\left(\frac{\partial}{\partial y}\right)=-\frac{\partial}{\partial x},\quad \varphi\left(\frac{\partial}{\partial z}\right)=0,\\
	\xi=\frac{\partial}{\partial z},\quad \eta=dz,\quad
	g=\frac{1}{2}exp(2z)(dx^2+dy^2)+dz^2,
	\end{align*}
	$(\varphi,\xi,\eta,g)$ is an almost contact metric structure on $M$. The matrix representation of $g$ with respect to $\frac{\partial}{\partial x}$, $\frac{\partial}{\partial y}$ and $\frac{\partial}{\partial z}$ is
	\begin{align*}
	(g_{ij})=\begin{pmatrix}
	\frac{1}{2}exp(2z) & 0 & 0 \\ 
	0 & \frac{1}{2}exp(2z) & 0 \\ 
	0 & 0 & 1
	\end{pmatrix}.
	\end{align*} Using the Koszul formula we have
	\begin{align}\label{3.33}
	\begin{gathered}
	\nabla_{\frac{\partial}{\partial x}}\frac{\partial}{\partial x}=-\frac{1}{2}exp(2z)\frac{\partial}{\partial z},\quad \nabla_{\frac{\partial}{\partial x}}\frac{\partial}{\partial y}=0,\quad \nabla_{\frac{\partial}{\partial x}}\frac{\partial}{\partial z}=\frac{\partial}{\partial x},\\
	\nabla_{\frac{\partial}{\partial y}}\frac{\partial}{\partial x}=0,\quad \nabla_{\frac{\partial}{\partial y}}\frac{\partial}{\partial y}=-\frac{1}{2}exp(2z)\frac{\partial}{\partial z},\quad \nabla_{\frac{\partial}{\partial y}}\frac{\partial}{\partial z}=\frac{\partial}{\partial y},\\
	\nabla_{\frac{\partial}{\partial z}}\frac{\partial}{\partial x}=\frac{\partial}{\partial x},\quad \nabla_{\frac{\partial}{\partial z}}\frac{\partial}{\partial y}=\frac{\partial}{\partial y},\quad \nabla_{\frac{\partial}{\partial x}}\frac{\partial}{\partial z}=0,
	\end{gathered}	
	\end{align}
	where $\nabla$ denotes the Levi-Civita connection of the Riemannian metric $g$. It is not hard to verify that the condition \eqref{3.1} for Kenmotsu manifold is satisfied. Hence, the manifold under consideration is a Kenmotsu manifold. By a straightforward calculation we have
	\begin{align}
	\nonumber &R\left(\frac{\partial}{\partial x},\frac{\partial}{\partial y}\right)\frac{\partial}{\partial z}=0,\quad R\left(\frac{\partial}{\partial x},\frac{\partial}{\partial z}\right)\frac{\partial}{\partial y}=0, \quad  R\left(\frac{\partial}{\partial x},\frac{\partial}{\partial z}\right)\frac{\partial}{\partial z}=-\frac{\partial}{\partial x},\\
	\nonumber &R\left(\frac{\partial}{\partial x},\frac{\partial}{\partial y}\right)\frac{\partial}{\partial x}=\frac{1}{2}exp(2z)\frac{\partial}{\partial y},\quad R\left(\frac{\partial}{\partial x},\frac{\partial}{\partial y}\right)\frac{\partial}{\partial y}=-\frac{1}{2}exp(2z)\frac{\partial}{\partial x},\\
	\nonumber &R\left(\frac{\partial}{\partial x},\frac{\partial}{\partial z}\right)\frac{\partial}{\partial x}=\frac{1}{2}exp(2z)\frac{\partial}{\partial z},\quad R\left(\frac{\partial}{\partial y},\frac{\partial}{\partial z}\right)\frac{\partial}{\partial x}=0,\\
	\label{3.34}&R\left(\frac{\partial}{\partial y},\frac{\partial}{\partial z}\right)\frac{\partial}{\partial y}=\frac{1}{2}exp(2z)\frac{\partial}{\partial z},\quad R\left(\frac{\partial}{\partial y},\frac{\partial}{\partial z}\right)\frac{\partial}{\partial z}=-\frac{\partial}{\partial y},
	\end{align}
	and applying these, we obtain
	\begin{align*}
	S\left(\frac{\partial}{\partial x},\frac{\partial}{\partial x}\right)=-exp(2z),\quad S\left(\frac{\partial}{\partial y},\frac{\partial}{\partial y}\right)=-exp(2z),\quad S\left(\frac{\partial}{\partial z},\frac{\partial}{\partial z}\right)=-2.
	\end{align*}
	Thus, the Ricci tensor satisfy  $S(X,Y)=-2ng(X,Y)$ for any $X,Y\in\Gamma(TM)$, that is, $M$ is Einstein. From \eqref{3.34}, we can easily shows that
	\begin{align}
	\label{3.35}R(X,Y)Z=-\{g(Y,Z)X-g(X,Z)Y\},
	\end{align}
	for any $X,Y,Z\in\Gamma(TM)$. Next consider a vector field 
	\begin{align}
	\label{3.36}V=a\left(y\frac{\partial}{\partial x}-x\frac{\partial}{\partial y}\right),
	\end{align}
	where $a\neq0$ is a constant. One can easily verify that $V$ has constant divergence. As a result of \eqref{3.33}, one can easily justify that
	\begin{align}
	\label{3.37}(\pounds_V g)(X,Y)=0,
	\end{align}
	for any $X,Y\in\Gamma(TM)$. Unifying \eqref{3.37} and \eqref{3.35}, we obtain that $g$ is a Riemann soliton, that is, \eqref{1.2} holds true with $V$ as in \eqref{3.36} and $\lambda=1$. Further \eqref{3.35} shows that $M$ is of constant negative curvature $-1$ and this verifies the Theorem \ref{t3.3}.
\end{ex}

\section{Riemann solitons and almost gradient Riemann solitons on $(\kappa,\mu)'$-almost Kenmotsu manifolds with $\kappa<-1$}
If the Reeb or characterstic vector field $\xi$ of an almost Kenmotsu manifold $M$ belonging to the $(\kappa,\mu)'$-nullity distribution (see \cite{Dileo Pastore2}), that is,
\begin{align}
\label{4.1}R(X,Y)\xi=\kappa\{\eta(Y)X-\eta(X)Y\}+\mu\{\eta(Y)h'X-\eta(X)h'Y\},
\end{align}
for some constants $\kappa$ and $\mu$, then $M$ is called $(\kappa,\mu)'$-almost Kenmotsu manifold. Classification of almost Kenmotsu manifold with $\xi$ belonging to the $(\kappa,\mu)'$-nulliy distribution have done by many geometers. For more detail, we refer to \cite{Dileo Pastore2,Prakasha, Vens3, Wang}. On $(\kappa,\mu)'$-almost Kenmotsu manifold, the following relation hold
\begin{align}
\label{4.2}h'^2=(\kappa+1)\varphi^2, &\quad \text{or} \quad h^2=(\kappa+1)\varphi^2,\\
&Q\xi=2n\kappa\xi.
\end{align}
Let $X\in Ker\eta$ be an eigenvector field of $h'$ orthogonal to $\xi$ with the corresponding eigenvalue $\theta$. As a result of \eqref{4.2}, we get $\theta^2=-(\kappa+1)$ and hence we have $\kappa\leq-1$. From \eqref{4.2} we remark that the tensor field $h'=0$ (equivalently, $h=0$) if and only if $\kappa=-1$, and $\kappa<-1$ if and only if $h'\neq 0$. As stated by Dileo and Pastore \cite{Dileo Pastore2} in Proposition 4.1 that, on a $(\kappa,\mu)'$-almost Kenmotsu manifold with $\kappa<-1$, we have $\mu=-2$. In this section, we paln to investigate the geometry of a Riemann soliton and gradient almost Riemann soliton on non-Kenmotsu $(\kappa,\mu)'$-almost Kenmotsu manifold. First, we reminisce the following result for our later use.
\begin{lem}\label{l4.1}
	In a $(\kappa,\mu)'$-almost Kenmotsu manifold $M$ with $\kappa<-1$, the Ricci operator satisfies
	\begin{align}
	\label{4.4}QX=-2nX+2n(\kappa+1)\eta(X)\xi-2nh',
	\end{align}
	where $h'\neq0$. Moreover, the scalar curvature $r$ of $M$ is $2n(\kappa-2n)$.
\end{lem}
Now, we prove the following outcome.
\begin{thm}\label{t4.2}
	Let $M$ be a non-Kenmotsu $(\kappa,\mu)'$-almost Kenmotsu manifold. If metric of $M$ is a Riemann soliton with $divV=constant$, then either the manifold is locally isometric to the Reimannian product $\mathbb{H}^{n+1}(-4)\times \mathbb{R}^n$ or the potential vector field $V$ is strict infinitesimal contact transformation. 
\end{thm}
\begin{proof}
	Contracting the equation \eqref{1.2} over $X$ and $W$, one can get
	\begin{align*}
	(\pounds_V g)(Y,Z)+\frac{2}{2n-1}S(Y,Z)+\frac{2}{2n-1}(2n\lambda+div V)g(Y,Z)=0.
	\end{align*}
	As a result of Lemma \ref{l4.1}, the above equation transform into
	\begin{align}
	\nonumber (\pounds_Vg)(Y,&Z)+\frac{2}{2n-1}\{(2n\lambda+div V-2n)g(Y,Z)\\
	\label{4.5}&+2n(\kappa+1)\eta(Y)\eta(Z)-2ng(h'Y,Z)\}=0.
	\end{align}
	Differentiating \eqref{4.5} covariantly along $X$, recalling \eqref{2.3} and $V$ has a constant divergence, we have
	\begin{align}
	\nonumber(\nabla_X\pounds_Vg)(Y,Z)=-&\frac{4n}{2n-1}(\kappa+1)\{g(X+h'X,Y)\eta(Z)+g(X+h'X,Z)\eta(Y)\\
	\label{4.6}&-2\eta(X)\eta(Y)\eta(Z)\}+\frac{4n}{2n-1}g((\nabla_Xh')Y,Z).
	\end{align}
	In \cite{Dileo Pastore2} Dileo and Pastore proved that on $M$ there holds 
	\begin{align}
	\nonumber g((\nabla_Xh')Y,Z)=g&((\kappa+1)X-h'X,Y)\eta(Z)+\eta(Y)g((\kappa+1)X-h'X,Z)\\
	\label{4.7}&-2(\kappa+1)\eta(X)\eta(Y)\eta(Z).
	\end{align}
	Making use of \eqref{4.6} into \eqref{2.6}, we entails that
	\begin{align}
	\label{4.8}(\pounds_V\nabla)(X,Y)=-\frac{4n}{2n-1}(\kappa+2)g(h'X,Y)\xi,
	\end{align}
	where we applied \eqref{4.7}. Taking covariant differentiation of \eqref{4.8} we get
	\begin{align*}
	\nonumber(\nabla_Y\pounds_V\nabla)(X,Z)=&-\frac{4n}{2n-1}(\kappa+2)g((\nabla_Yh')Y,Z)\xi\\
	&-\frac{4n}{2n-1}(\kappa+2)g(h'X,Z)(Y-\eta(Y)\xi+h'Y).
	\end{align*}
	Putting the foregoing equation into \eqref{2.7} yields
	\begin{align}
	\nonumber (\pounds_V &R)(X,Y)Z=\frac{4n}{2n-1}\{g((\nabla_Yh')X,Z)\xi-g((\nabla_Xh')Y,Z)\xi\\
	\label{4.9}&+g(h'X,Z)(Y-\eta(Y)\xi+h'Y)-g(h'Y,Z)(X-\eta(X)\xi+h'X)\}.
	\end{align}
	Contracting the equation \eqref{4.9} with respect to $X$ and recalling \eqref{4.7} we easily obtain
	\begin{align}
	\label{4.10}(\pounds_VS)(Y,Z)=-\frac{8n^2}{2n-1}(\kappa+2)g(h'Y,Z).
	\end{align}
	From \eqref{4.4}, the Ricci tensor can be written as
	\begin{align*}
	S(Y,Z)=-2ng(Y,Z)+2n(\kappa+1)\eta(Y)\eta(Z)-2ng(h'Y,Z).
	\end{align*}
	Lie-derivative of this equation along potential vector field $V$, utilization of \eqref{2.3} and \eqref{4.5} yields
	\begin{align}
	\nonumber (\pounds_V S)(Y,Z)&=\frac{4n}{2n-1}\{(2n\lambda+div V+2n\kappa)g(Y,Z)-2(\kappa+1)(2n\lambda+div V\\
	\nonumber &+2n\kappa)\eta(Y)\eta(Z)+(2n\lambda+div V-4n)g(h'Y,Z)\}+2n(\kappa+1)\\
	\label{4.11}&\{\eta(Y)g(\pounds_V\xi,Z)+\eta(Z)g(\pounds_V\xi,Y)\}-2ng((\pounds_Vh')Y,Z).
	\end{align}
	Putting \eqref{4.11} together with \eqref{4.10}, we get
	\begin{align}
	\nonumber g((\pounds_Vh')Y,Z)=&\frac{2}{2n-1}\{(2n\lambda+div V+2n\kappa)g(Y,Z)-2(\kappa+1)(2n\lambda+div V\\
	\nonumber &+2n\kappa)\eta(Y)\eta(Z)+(2n(\kappa+2)+2n\lambda+div V-4n)g(h'Y,Z)\}\\
	\label{4.12}&+(\kappa+1)\{\eta(Y)g(\pounds_V\xi,Z)+\eta(Z)g(\pounds_V\xi,Y)\}.
	\end{align}
	Note that by switching $Y=Z=\xi$ in \eqref{4.5} we obtain $\eta(\pounds_V\xi)-\frac{1}{2n-1}(2n\lambda+div V+2n\kappa)=0$. Applying this equation and substituting $Y=Z=\xi$ in \eqref{4.12} one can get $(2n\lambda+div V+2n\kappa)=0$ by reason of $\kappa<-1$ and eventually we have
	\begin{align*}
	(\pounds_Vh')Y=(\kappa+1)\{\eta(Y)\pounds_V\xi+g(\pounds_V\xi,Y)\xi\}.
	\end{align*}
	Inserting $Y$ by $\xi$ in the foregoing equation and utilization of $(2n\lambda+\Delta u+2n\kappa)=0$ gives
	\begin{align}
	\label{4.13}h'\pounds_V\xi=-(\kappa+1)\pounds_V\xi.
	\end{align}
	In view of \eqref{4.2}, $(2n\lambda+div V+2n\kappa)=0$ and $\kappa<-1$, the action of $h'$ on the above equation gives $h'\pounds_V\xi=\pounds_V\xi$. This together with \eqref{4.13} yields 
	\begin{align*}
	(\kappa+2)\pounds_V\xi=0.
	\end{align*}
	\par Suppose $\kappa=-2$, then it follows from Proposition 4.1 and Corollary 4.2 of Dielo and Pastore \cite{Dileo Pastore2} that a non-Kenmotsu $(\kappa,\mu)'$-almost Kenmotsu manifold is locally isometric to the Riemannian product $\mathbb{H}^{n+1}(-4)\times \mathbb{R}^n$. If $\kappa\neq-2$, then we have $\pounds_V\xi=0$. As a result of $2n\lambda+div V+2n\kappa=0$ and $(\pounds_Vg)(X,\xi)=0$ we obtain
	\begin{align*}
	(\pounds_V\eta)X=(\pounds_V g)(X,\xi)+g(X,\pounds_V\xi)=0.
	\end{align*}
	This means potential vector field $V$ is strict infinitesimal contact transformation. This finishes the proof.
\end{proof}
As a result of Lemma \ref{l4.1}, the relation \eqref{3.20} and \eqref{3.19}, we obtain the following fruitful result.
\begin{thm}
	Let  $M$ be a non-Kenmotsu $(\kappa,\mu)'$-almost Kenmotsu manifold which admits a gradient almost Riemann soliton. Then, the soliton is expanding with $\lambda=\frac{6n-2}{2n-1}$ and $M$ is locally isometric to the Riemannian product $\mathbb{H}^{n+1}(-4)\times\mathbb{R}^n$. Moreover, the potential vector field is
	tangential to the Euclidean factor $\mathbb{R}^n$.
\end{thm}
\begin{proof}
	With the help of \eqref{4.4} we have
	\begin{align*}
	\nonumber (\nabla_Y Q)X-(\nabla_X Q)Y=&-2n((\nabla_Yh')X-(\nabla_Xh')Y)\\
	&-2n(\kappa+1)(\eta(Y)(X+h'X)-\eta(X)(Y+h'Y)).
	\end{align*}
	Utilization of the above equation in \eqref{3.19} gives
	\begin{align}
	\nonumber R(X,Y)Du=&-\frac{2n}{2n-1}\{((\nabla_Yh')X-(\nabla_Xh')Y)\\
	\nonumber &+(\kappa+1)(\eta(Y)(X+h'X)-\eta(X)(Y+h'Y))\}\\
	\label{4.14}&+\frac{1}{2n-1}\{Y(2n\lambda+\Delta u)X-X(2n\lambda+\Delta u)Y\}
	\end{align}
	Using $\xi$ in place of $X$ in \eqref{4.14} we get an equation, then taking inner product of the resulting equation with $\xi$ gives
	\begin{align}
	\label{4.15}g(R(\xi,Y)Du,\xi)=\frac{1}{2n-1}\{Y(2n\lambda+\Delta u)-\xi(2n\lambda+\Delta u)\eta(Y)\}.
	\end{align}
	As a result of \eqref{4.1}, we get
	\begin{align}
	\label{4.16}g(R(\xi,Y)\xi,Du)=\kappa\{\xi(u)\eta(Y)-Y(u)\}+2g(h'Du,Y),
	\end{align}
	where we used $\mu=-2$. Comparing \eqref{4.15} with \eqref{4.16} yields that
	\begin{align}
	\label{4.17}\frac{1}{2n-1}D(2n\lambda+\Delta u)=\kappa Du-\kappa \xi(u)\xi-2h'Du+\frac{1}{2n-1}\xi(2n\lambda+\Delta u)\xi.
	\end{align}
	According to Corollary 4.1 of Delio and Pastore \cite{Dileo Pastore2}, and Lemma 3.4 of Wang and Liu \cite{Wang Liu}, we have that $tr(\nabla_X h')=0$ and $(div h')X=2n(\kappa+1)\eta(X)$. Applying these equation and contracting \eqref{4.14} over $X$ gives that
	\begin{align*}
	S(Y,Du)=\frac{2n}{2n-1}Y(2n\lambda+\Delta u).
	\end{align*}
	As a result of \eqref{4.4}, the above equation transform into
	\begin{align}
	\label{4.18}\frac{1}{2n-1}D(2n\lambda+\Delta u)=-Du+(\kappa+1)\xi(u)\xi-h'Du.
	\end{align}
	By virtue of \eqref{4.17} and \eqref{4.18} we obtain
	\begin{align*}
	(\kappa+1)Du-(2\kappa+1)\xi(u)\xi-h'Du+\frac{1}{2n-1}\xi(2n\lambda+\Delta u)=0.
	\end{align*}
	Remembering the assumption $\kappa<-1$ and equation \eqref{4.2}, the action of $h'$ on above equation gives that 
	\begin{align}
	\label{4.19}h'Du=-Du+\xi(u)\xi.
	\end{align}
	Again making use of \eqref{4.2}, the action $h'$ on \eqref{4.19} yields $(\kappa+1)(Du-\xi(u)\xi)=h'Du$. This together with \eqref{4.19} gives that
	\begin{align}
	\label{a}(\kappa+2)(Du-\xi(u)\xi)=0.
	\end{align}
	Thus, we have either $\kappa=-2$ or $Du=\xi(u)\xi$. Now, we prove that the second case can not occur. In fact, if we assume that the second case is true, that is, $Du=\xi(u)\xi$, then covariant derivative of this along $X$ gives
	\begin{align*}
	\nabla_X Du=X(\xi(u))\xi+\xi(u)\{X-\eta(X)\xi+h'X\}.
	\end{align*}
	With the aid of above equation, \eqref{3.20} becomes
	\begin{align*}
	\frac{1}{2n-1}QX=-(\xi(u)+\frac{1}{2n-1}(2n\lambda+\Delta u))X+(\xi(u)\eta(X)-X(\xi(u)))\xi-\xi(u)h'X.
	\end{align*}
	Utilization of \eqref{4.4} in the foregoing equation furnishes
	\begin{align}
	\nonumber\left(\xi(u)+\frac{1}{2n-1}(2n\lambda+\Delta u-2n)\right)X+\left(\xi(u)-\frac{2n}{2n-1}\right)h'X\\
	\label{4.20}+\left(\frac{2n}{2n-1}(\kappa+1)\eta(X)+X(\xi(u))-\xi(u)\eta(X)\right)\xi=0.
	\end{align}
	Contracting \eqref{4.20} with respect to $X$ and recalling \eqref{2.4} we obtain
	\begin{align*}
	2n(\kappa+1)\left(\xi(u)-\frac{2n}{2n-1}\right)=0,
	\end{align*}
	and this exhibits that $\xi(u)=\frac{2n}{2n-1}$ because of $\kappa<-1$. Putting $\xi(u)=\frac{2n}{2n-1}$ into \eqref{4.20} we obatin
	\begin{align}
	\label{4.21}(2n\lambda+\Delta u)X+2n\kappa\eta(X)\xi=0.
	\end{align}
	Considering $X$ in \eqref{4.21} is orthogonal to $\xi$, we have $2n\lambda+\Delta u=0$. Thus, \eqref{4.21} becomes $2n\kappa\eta(X)\xi=0$. It follows that $\kappa=0$, this contradicts our assumption $\kappa<-1$. Therefore, we obtain from \eqref{a} that $\kappa=-2$. In view of $\kappa=-2=\mu$, on the basis of results of Dileo and Pastore \cite{Dileo Pastore2} we obtain that $M$ is locally isometric to the Riemannian product $\mathbb{H}^{n+1}(-4)\times\mathbb{R}^n$. Employing $\kappa=-2$ and \eqref{4.19} in \eqref{4.17} we obtain
	\begin{align}
	\label{4.23}D(2n\lambda+\Delta u)=\xi(2n\lambda+\Delta u).
	\end{align}
	Utilizing \eqref{4.23} and \eqref{4.19} in \eqref{4.18} provides
	\begin{align}
	\label{4.24}\xi\left(\frac{1}{2n-1}(2n\lambda+\Delta u)+2u\right)=0.
	\end{align}
	In the framework of \eqref{4.4}, \eqref{3.20} and \eqref{4.24} we acquire that
	\begin{align}
	\label{4.25}2n\lambda+\Delta u=4n+\frac{1}{2}\xi(\xi(2n\lambda+\Delta u)).
	\end{align}
	\par From \eqref{4.2} we have $h'^2=-\varphi^2$. Now, we shall indicate by $[1]'$ and $[-1]'$ the distribution of the eigen vectors of $h'$ orthogonal to $\xi$ with eigen values $1$ and $-1$ respectively, and also we may consider a local orthonormal $\varphi$-frame $\{E_i,\varphi E_i,\xi\}_{i=1}^{n}$ with $E_i\in[1]'$ and $\varphi E_i\in [-1]'$. As a result of \eqref{4.19}, we easily find that $Du$ has no components on the distribution $[1]'$. Thus, we write $Du=\sum_{i=n}^n\gamma_i\varphi E_i+\xi(u)\xi$, where $\{\gamma_i\}_{i=1}^{n}$ are smooth function on $M$. Applying this and \eqref{4.24} in \eqref{3.20} we obtain that
	\begin{align*}
	\frac{1}{2n-1}QX=&\frac{1}{2n-1}\left(\frac{1}{2}\xi(2n\lambda+\Delta u)-(2n\lambda+\Delta u)\right)X-\sum_{i=1}^n X(\gamma_i)\varphi E_i\\
	&-\sum_{i=1}^n\gamma_i \nabla_X\varphi E_i+\frac{1}{2(2n-1)}\xi(2n\lambda+\Delta u)h'X\\
	&+\frac{1}{2(2n-1)}\left(X(\xi(2n\lambda+\Delta u))-\xi(2n\lambda+\Delta u)\eta(X)\right)\xi.
	\end{align*}
	By virtue of \eqref{4.4}, the above equation transform into
	\begin{align}
	\nonumber \frac{1}{2n-1}&\left(\frac{1}{2}\xi(2n\lambda+\Delta u)-(2n\lambda+\Delta u)+2n\right)X-\sum_{i=1}^n X(\gamma_i)\varphi E_i-\sum_{i=1}^n\gamma_i\nabla_X\varphi E_i\\
	\nonumber &+\frac{1}{2n-1}\left(\frac{1}{2}\xi(2n\lambda+\Delta u)+2n\right)h'X+\frac{1}{2(2n-1)}\{X(\xi(2n\lambda+\Delta u))\\
	\label{4.26}&-\xi(2n\lambda+\Delta u)\eta(X)+4n\eta(X)\}\xi=0.
	\end{align}
	According to the proof of Proposition 4.1 of \cite{Dileo Pastore2}, we have that $\nabla_{E_j}\varphi E_i\in[-1]'$ for any $E_i\in[1]'$, $1\leq j\leq n$. Consequently, using $E_j\in[1]'$ in place of $X$ in \eqref{4.26} provides
	\begin{align}
	\label{4.27}(2n\lambda+\Delta u)=\xi(2n\lambda+\Delta u)+4n.
	\end{align}
	Unifying \eqref{4.25} with \eqref{4.27} yields that $2n\lambda+\Delta u=4n$. Now, contracting \eqref{3.20} over $X$ gives
	\begin{align*}
	\Delta u=-\frac{4n^2}{2n-1},
	\end{align*}
	where we applied $\kappa=-2$ and $r=2n(\kappa-2n)$.	Combining the above equation with $2n\lambda+\Delta u=4n$ yields that $\lambda=\frac{6n-2}{2n-1}>0$, and this means that the gradient almost Riemann soliton is expanding. Eventually, utilization of $2n\lambda+\Delta u=4n$ in \eqref{4.24} gives $\xi(u)=0$ and thus from \eqref{4.19} we have $h'Du=-Du$. According to Theorem 4.2 of \cite{Dileo Pastore2}, the factor $\mathbb{R}^n$ in the product space $\mathbb{H}^{n+1}(-4)\times\mathbb{R}^n$ is the integral submanifold of the distribution $[-1]'$. Therefore, $h'Du=-Du$ implies that the gradient of the potential function $Du$ is tangential to the Euclidean factor $\mathbb{R}^n$.
\end{proof}
\begin{rem}
	Theorem~3.1 of Wang \cite{Wang} is a direct corollary of the above Theorem~4.3.
\end{rem}
Before closing this section, we present an example of 3-dimensional $(\kappa,\mu)'$-almost Kenmotsu manifolds admitting a Riemann soliton.

\begin{ex}
	We consider a 3-dimensional manifold
	\begin{align*}
	M=\{(x,y,z)\in \mathbb{R}^3, z\neq0\},
	\end{align*}
	and linear independently vector fields
	\begin{align*}
	e_1=-\frac{\partial}{\partial x}+2y\frac{\partial}{\partial y}-\frac{\partial}{\partial z},\quad e_2=\frac{\partial}{\partial y},\quad e_3=\frac{\partial}{\partial z}.
	\end{align*}
	One can easily check that
	\begin{align*}
	[e_1,e_2]=-2e_2,\quad [e_2,e_3]=0,\quad [e_1,e_3]=0.
	\end{align*}
	On $M$ we define a (1,1)-tensor field $\varphi$ by $\varphi(e_1)=0$, $\varphi(e_2)=e_3$ and $\varphi(e_3)=-e_2$, and we define a Riemannian metric $g$ such that $g(e_i,e_j)=\delta_{ij}$, $1\leq i,j\leq 3$. We denote by $\xi=e_1$ and $\eta$ its dual 1-form with respect to the metric $g$. 
	\par On the other hand, it is not hard to justify that $M$ with the structure $(\varphi,\xi,\eta,g)$ is an almost Kenmotsu structure, which is not Kenmotsu, since the operator $h'$ does not vanish. In fact, it is given by
	\begin{align*}
	h'(e_1)=0,\quad h'(e_2)=e_2,\quad h'(e_3)=-e_3.
	\end{align*}
	Utilization of Koszul formula gives
	\begin{align}
	\begin{gathered}\label{4.28}
	\nabla_{e_1}e_1=0,\quad \nabla_{e_1}e_2=0,\quad \nabla_{e_1}e_3=0,\\
	\nabla_{e_2}e_1=2e_2,\quad \nabla_{e_2}e_2=-2e_1,\quad \nabla_{e_2}e_3=0,\\
	\nabla_{e_3}e_1=0,\quad \nabla_{e_3}e_2=0,\quad \nabla_{e_3}e_3=0,
	\end{gathered}
	\end{align}
	where $\nabla$ denotes the Levi-Civita connection of the Riemannian metric $g$. By a straightforward calculation we have
	\begin{align}\label{4.29}
	\begin{gathered}
	 R(e_1,e_2)e_1=4e_2,\quad R(e_1,e_2)e_2=-4e_1,\quad R(e_1,e_2)e_3=0,\\
	 R(e_1,e_3)e_1=0,\quad R(e_1,e_3)e_2=0,\quad R(e_1,e_3)e_3=0,\\
	 R(e_2,e_3)e_1=0,\quad R(e_2,e_3)e_2=0,\quad R(e_2,e_3)e_3=0.
	\end{gathered}
	\end{align}
	With the help of the expressions of the curvature tensor, we conclude that the
	characteristic vector field $\xi$ belongs to the $(\kappa,\mu)'$-nullity distribution with $\kappa=-2$ and $\mu=-2$.
	\par Let us consider a vector field 
	\begin{align}\label{4.30}
	V=e^{-2x}\frac{\partial}{\partial y}+4(x-z)\frac{\partial}{\partial z}.
	\end{align}
	One can easily check that $div V=-4$, a constant. As a result of \eqref{4.28}, one can get
	\begin{align}\label{4.31}
	\begin{gathered}
	(\pounds_V g)(e_1,e_1)=0, \quad (\pounds_V g)(e_2,e_2)=0, \quad (\pounds_V g)(e_3,e_3)=-8,\\
	(\pounds_V g)(e_1,e_2)=0,\quad (\pounds_V g)(e_1,e_3)=0,\quad (\pounds_V g)(e_2,e_3)=0
	\end{gathered}
	\end{align}
	In view of \eqref{4.31} and \eqref{4.29}, one can easily verify that
	\begin{align*}
	2R(e_i,e_j,e_k,e_l)+4(g\mathbin{\bigcirc\mspace{-15mu}\wedge\mspace{3mu}} g)(e_i,e_j,e_k,e_l)+(g \mathbin{\bigcirc\mspace{-15mu}\wedge\mspace{3mu}} \pounds_V g)(e_i,e_j,e_k,e_l)=0,
	\end{align*}
	for $1\leq i,j,k,l\leq 3$. Thus $g$ is a Riemann soliton with the soliton vector field $V$ as given in \eqref{4.30} and $\lambda=4$. According to Dileo and Pastore \cite{Dileo Pastore2}, we obtain that $(-2,-2)'$-almost Kenmotsu manifold $M$ is locally isometric to the product $\mathbb{H}^2(-4)\times\mathbb{R}$. This verifies Theorem~\ref{t4.2}.
\end{ex}


\begin{thebibliography}{1}
	
	\bibitem{Blair} D.E. Blair, \textit{Riemannian geometry of contact and symplectic manifolds.} In: Progress in Mathematics, 203. Birkh\"auser, Boston (2010).
	
	\bibitem{Dileo Pastore} G. Dileo, A. M. Pastore, \textit{Almost Kenmotsu manifolds and local symmetry}, Bull. Belg. Math. Soc. Simon Stevin. 14(2) (2007), 343-354.
	
	\bibitem{Dileo Pastore2} G. Dileo, A. M. Pastore, \textit{Almost Kenmotsu manifolds and nullity distributions}, J. Geom. 93(1-2) (2009), 46-61.
	
	\bibitem{Dev} M.N. Devaraja, H.A. Kumara, V. Venkatesha, \textit{ Riemann
		soliton within the framework of contact geometry}, Quaestiones Mathematicae, (2020) DOI:	10.2989/16073606.2020.1732495
	
	\bibitem{Ghosh3} A. Ghosh, \textit{Kenmotsu 3-metric as a Ricci soliton}, Chaos Solitons \& Fractals, 44 (2011), 647-650.
	
	\bibitem{Ghosh2} A. Ghosh, \textit{An $\eta$-Einstein Kenmotsu metric as a Ricci soliton}, Publ. Math. Debrecen, 82 (2013), 691-598.
	
	\bibitem{Ghosh} A. Ghosh, \textit{Ricci soliton and Ricci almost soliton within the framework of Kenmotsu manifold}, Carpathian Math. Publ. 11(1) (2019), 59-69.
	
	\bibitem{Janssens Vanhecke} D. Janssens, L. Vanhecke, \textit{Almost contact structures and curvature tensors}, Kodai Math. J. 4(1) (1981), 1-27.
	
	\bibitem{Hamilton} R.S, Hamilton, \textit{The Ricci flow on surfaces. Mathematics and general relativity}. Contemp. Math., Amer. Math. Soc. 71 (1988), 237-262.
	
	\bibitem{Hirica} I.E. Hiri\u{c}a, C. Udri\c{s}ste, \textit{Ricci and Riemann solitons}, Balkan J. Geom. Applications. 21(2) (2016), 35-44.
	
	\bibitem{Hirica2} I.E. Hiri\u{c}a, C. Udri\c{s}ste, \textit{Basic evolution PDE’s in Riemannian Geometry}, Balkan J.Geom.Appl. 17(1) (2012), 30-40.
	
	\bibitem{Kenmotsu} K. Kenmotsu, \textit{A class of almost contact Riemannian manifolds}, T\^ohoku Math. J. 24(1) (1972), 93-103.
	
	\bibitem{Prakasha} D.G. Prakasha, P. Veeresha, Venkatesha, \textit{The Fischer–Marsden conjecture on non-Kenmotsu $(\kappa,\mu)'$-almost Kenmotsu manifolds}, J. Geom. 110 (1) (2019), https://doi.org/10.1007/s00022-018-0457-8.
	
	\bibitem{Step} S.E. Stepanov, I.I. Tsyganok, \textit{The theory of infinitesimal harmonic transformations and its applications to the global geometry of Riemann solitons}, Balk. J. Geom. Appl. 24 (2019), 113–121.
	
	\bibitem{Tashiro} Y. Tashiro, \textit{Complete Riemannian manifolds and some vector fields}, Trans. Amer. Math. Soc. 117 (1965), 251-275.
	
	\bibitem{Udriste} C. Udri\c{s}te, \textit{Riemann flow and Riemann wave}, Ann. Univ. Vest, Timisoara. Ser. Mat.-Inf. 48(1-2) (2010), 265-274.
	
	\bibitem{Udriste2} C. Udri\c{s}te, \textit{Riemann flow and Riemann wave via bialternate product Riemannian metric}. preprint, arXiv.org/math.DG/1112.4279v4 (2012).
	
	\bibitem{Vens} Venkatesha, D. M. Naik, H. A. Kumara, \textit{$*$-Ricci soliton and gradient almost $*$-Ricci soliton on Kenmotsu manifolds}, Mathemtica Slovoca (accepted).
	
	\bibitem{Vens3} V. Venkatesha, H. A. Kumara, \textit{Gradient  $\rho$-Einstein soliton on almost Kenmotsu manifolds}, Ann. Univ. Ferrara. 65(2) (2019), 375-388.
	
	\bibitem{Wang} Y. Wang, \textit{Gradient almost Ricci soliton on two classes of almost Kenmotsu manifolds}, J. Korean Math. Soc. 53(5) (2016), 1101-1114.
	
	\bibitem{Wang Liu}  Y. Wang, X. Liu, \textit{Locally symmetric CR-integrable almost Kenmotsu manifolds}, Mediterr. J. Math. 12(1) (2015), 159-171.
	
	\bibitem{Yano} K. Yano, Integral formulas in Riemannian geometry, Marcel Dekker, New York, 1970.
	
\end{thebibliography}
\end{document}